\documentclass[12pt]{amsart}

\usepackage[utf8]{inputenc}
\usepackage{amsfonts}
\usepackage{amsthm}
\usepackage{amssymb}
\usepackage{amsmath}
\usepackage{amscd}
\usepackage{latexsym,dsfont}
\usepackage{bbm}
\usepackage{mathrsfs}

\usepackage{times}
\usepackage{microtype}
\usepackage{setspace}
\usepackage[margin=1.2in]{geometry}

\usepackage{cite}

\usepackage[colorlinks=true, pdfstartview=FitV, linkcolor=blue,
citecolor=blue, urlcolor=blue]{hyperref}

\newtheorem{theorem}{Theorem}[section]
\newtheorem{lemma}[theorem]{Lemma}
\newtheorem{prop}[theorem]{Proposition}
\newtheorem{cor}[theorem]{Corollary}

\newtheorem{defn}[theorem]{Definition}

\theoremstyle{definition}

\theoremstyle{remark}
\newtheorem{remark}[theorem]{Remark}

\numberwithin{equation}{section}
 \allowdisplaybreaks

\def\Xint#1{\mathchoice
   {\XXint\displaystyle\textstyle{#1}}%
   {\XXint\textstyle\scriptstyle{#1}}%
   {\XXint\scriptstyle\scriptscriptstyle{#1}}%
   {\XXint\scriptscriptstyle\scriptscriptstyle{#1}}%
   \!\int}
\def\XXint#1#2#3{{\setbox0=\hbox{$#1{#2#3}{\int}$}
     \vcenter{\hbox{$#2#3$}}\kern-.5\wd0}}

\def\avgint{\Xint-}

\DeclareMathOperator{\Div}{div}
\DeclareMathOperator{\supp}{supp}

\DeclareMathOperator{\sgn}{sgn}
\DeclareMathOperator{\grad}{\nabla}

\newcommand{\op}{{\mathrm{op}}}

\DeclareMathOperator{\lip}{\mathrm{Lip}}

\newcommand{\N}{\mathbb N}

\newcommand{\R}{\mathbb{R}}
\newcommand{\rn}{{\mathbb{R}^n}}

\newcommand{\loc}{\mathrm{loc}}

\newcommand{\vecg}{\mathbf g}

\newcommand{\vect}{\mathbf t}

\title[Bounded solutions and Orlicz gain]
{Bounded solutions of degenerate elliptic equations\\ with an Orlicz-gain Sobolev inequality }

\author{David Cruz-Uribe OFS, Sullivan MacDonald, and Scott Rodney}

\address{David Cruz-Uribe, OFS \\
Dept. of Mathematics \\
University of Alabama \\
 Tuscaloosa, AL 35487, USA}

\email{dcruzuribe@ua.edu}

\address{ Sullivan F.~MacDonald\\
Dept. of Mathematics \\
University of Toronto \\
Toronto, Ontario \\
Canada, M5S 2E4}

\email{sullivan.macdonald@mail.utoronto.ca}

\address{Scott Rodney\\
Dept. of Mathematics, Physics and Geology \\ 
Cape Breton University \\
Sydney, NS B1Y3V3, CA} 

\email{scott\_rodney@cbu.ca}

\thanks{The first author is partially supported by a Simons Foundation
  Travel Support for Mathematicians Grant and by NSF Grant DMS-2349550. At the time of this project, the second author was supported by a Masters' scholarship from McMaster University.  The third author  is partially supported by an NSERC development grant.  The first and third authors would also like to thank TUBITAK, the Scientific and Technological Research Council of
    T\"urkiye, for their support through a 2501 Joint Research Program grant, Yusuf Zeren, Y\i ld\i z Technical University, principal investigator.
    Finally, the authors would like to thank the referees for their thorough and detailed reports which led to a number of improvements in the paper.}

\keywords{degenerate elliptic equations, Sobolev inequalities, $p$-Laplacian}

\subjclass{35B65, 35J62, 35J70, 42B37, 46E30}

\begin{document}
\begin{abstract}
We consider the boundedness and exponential integrability of solutions to the Dirichlet problem for the degenerate quasilinear elliptic equation 
\[ -v^{-1}\Div(|\sqrt{Q}\grad u|^{p-2}Q\grad u)=f|f|^{p-2} -v^{-1}\Div(v|g|^{p-2}g{\bf t}),  \quad 1<p<\infty, \]
assuming that there is a Sobolev inequality of the form 
\[ \|\varphi\|_{L^N(v,\Omega)}\leq S_N\|\sqrt{Q} \nabla\varphi\|_{L^p(\Omega)}, \]
where $N$ is a power function of the form $N(t)=t^{\sigma p}$, $\sigma\geq 1$, or a Young function of the form $N(t)=t^p\log(e+t)^\sigma$, $\sigma>1$.  In our results we study the interplay between the Sobolev inequality and the regularity assumptions needed on $f$ and $g$ to prove that the solution is bounded or is exponentially  integrable.  Our results generalize those previously proved in~\cite{MR4280269,10.48550/arxiv.2210.12441}.
\end{abstract}

\maketitle

 \section{Introduction and main results}

In this paper we explore the regularity of solutions to a family of degenerate quasilinear elliptic equations, assuming the existence of a Sobolev inequality adapted to the geometry of the equation.  Fix a bounded domain $\Omega$ and  $1<p<\infty$.  Let $Q$ be an $n\times n$ positive semi-definite, measurable  matrix function, and  let $v$ be a weight.    We are interested in the behavior of solutions of the Dirichlet problem for the Poisson equation for the degenerate $p$-Laplacian:
\begin{align}\label{dirichlet-problem} 
\left\{\begin{array}{rcl}
    -v^{-1}\Div(|\sqrt{Q}\grad u|^{p-2} Q\grad u) &=& f|f|^{p-2}-v^{-1}\Div(v|g|^{p-2}g{\bf t}) \textrm{ in } \Omega,\\
    u &=& 0 \textrm{ in } \partial \Omega,
\end{array}\right.
\end{align}
where $f,g$ are real-valued measurable functions on $\Omega$, and ${\bf t}$ is $\mathbb{R}^n$-valued and defines a degenerate subunit vector field on $\Omega$.  (See Section~\ref{section:prelim} below for precise definitions.)  Of particular interest is the case when $Q$ is not uniformly elliptic: that is, it has upper and lower eigenvalues that are not necessarily bounded or bounded away from $0$ in $\Omega$. In particular, we will consider the case where for all $\xi \in \R^n$ and a.e.~$x\in \Omega$,
\[  0 \leq |\sqrt{Q(x)}\xi|^p \leq kv(x)|\xi|^p.  \]

Starting with the seminal paper by Fabes, Kenig and Serapioni~\cite{MR643158}, this equation has been studied by a number of authors, both in the linear case when $p=2$ and the quasilinear case when $1<p<\infty$.    See, for instance,~\cite{CW,Korobenko:2016ue,MR3011287,MR2771262,MR0839035,MR0847996,MR805809,MR1354890,MR1207810,MR2944065,MR2341517,MR2435212,deGiorgi-preprint,korobenko2021,korobenko2024}.  In these works the authors generally considered a particular geometric setting and showed that a Sobolev inequality configured to that specific geometry was satisfied, and then proved their results.  

More recently, beginning with the work of Sawyer and Wheeden~\cite{SW2,MR2204824}, a more abstract approach has developed, wherein it is assumed that a degenerate Sobolev inequality of the form
\begin{equation} \label{eqn:intro-sobolev}
 \|\varphi\|_{L^{\sigma p}(v,\Omega)} \leq C\|\sqrt{Q}\grad \varphi\|_{L^p(\Omega)} 
 \end{equation}
holds for some $\sigma\geq 1$ and all $\varphi\in \lip_0(\Omega)$.  This is then used to deduce properties of solutions of the Dirichlet problem.  This approach has been continued by the authors and their collaborators:  see~\cite{MR4280269,CRR1,CRR2,MRW,MR,R,10.48550/arxiv.2210.12441, macdonald2023,MR4069608,MR3388872,MR3095112}.  There are drawbacks to this approach, the most obvious being that it requires proving the existence of a Sobolev inequality in a particular geometric setting before the results can be applied.  Such inequalities have been proved in a variety of settings: see, for instance,~\cite{MR643158,MR805809,DCU-FED-SR,Korobenko:2016ue,korobenko2024} and the forthcoming paper~\cite{DCU-FED-SR-2}. Also, see Section~\ref{section:geometric} below.   On the other hand, this approach has the advantage that it makes clear the precise role played by the Sobolev inequality in the theory.  We  further explore this role  in a recent paper on the existence of solutions to linear degenerate elliptic equations~\cite{tubitak1}.

In this paper we consider three kinds of Sobolev inequalities:  first, a Sobolev inequality with a power gain (that is, as in inequality~\eqref{eqn:intro-sobolev}).  This case has been examined previously in~\cite{R,10.48550/arxiv.2210.12441} when $g=0$, and we include it partly for completeness.  Second, we consider a Sobolev inequality where the gain is in an Orlicz space that lies between $L^p$ and $L^{q}$ for any $q>p$--that is, 
\begin{align}\label{sobolev} 
\|\varphi \|_{L^N(v,\Omega)} \leq S_N \|\grad \varphi\|_{QL^p(\Omega)},
\end{align}
where $N(t)=t^p\log(e+t)^\sigma$, $\sigma>0$.  There is recent interest in such Sobolev inequalities for their applications to infinitely degenerate elliptic equations~\cite{korobenko2024,Korobenko:2016ue,korobenko2021,MR4331593}, since the underlying geometry in this case is nondoubling, and it has been shown that the existence of a Sobolev inequality with power gain~\eqref{eqn:intro-sobolev} is in some sense equivalent to the underlying geometry being doubling. See~\cite{MR3359590,MR4277804}.   Finally, we consider what can be accomplished assuming a Sobolev inequality without gain, i.e.,  where $N(t)=t^p$, which corresponds to a very rough underlying geometry.  To the best of our knowledge, little or no work has been done with this assumption.  It is of interest since it shows what can be found in common among all spaces that support a  Sobolev inequality with some kind of gain.  

To provide some context and motivation for our main theorems, we first recall a classical result.    If $u\in W^{1,2}_0(\Omega)$ is a solution to the equation $-\Div(Q\grad u)=f-\Div(\vecg)$, where $Q$ is uniformly elliptic, then $u$ is bounded if $f\in L^q(\Omega)$ and $\vecg \in L^{2q}(\Omega)$, whenever $q>\frac{n}{2}$.  See~\cite[Theorem~8.16]{GT}, which is attributed to Trudinger~\cite{MR369884}.    Versions of this result (e.g.,~\cite[Theorem~8.15]{GT}) were proved by Ladyzehnskaya and Ural$'$tseva~\cite{MR244627}, and Stampacchia~\cite{MR192177,MR251373}.  Related results which assume stronger boundary conditions (e.g., $u$ is continuous up to the boundary, see~\cite[Theorem~9.1]{GT}) were proved by Alexsandrov and Bakelman. (See the notes to~\cite[Chapter~9]{GT} for details and further references.)
A key observation is that in this classical result, the critical index is $\frac{n}{2}=(\frac{n}{n-2})'$, that is, it is the dual of the gain in the classical Sobolev inequality.  In our setting of degenerate elliptic equations, we consider the interplay between the gain in the Sobolev inequality in different scales and the assumptions on $f$ and $\vecg$ to ensure that weak solutions are bounded.  We first consider gain in the scale of Lebesgue spaces, that is, a power gain.

\begin{theorem} \label{thm:power-gain}
Given a bounded domain $\Omega\subset \mathbb{R}^n$ and $1<p<\infty$, suppose $Q$ and $v$ are such that 
$v\in L^1(\Omega)$, $|Q|_\op^{\frac{p}{2}} \leq kv$, and $\Omega$ supports the Sobolev inequality \eqref{sobolev} with $N(t)=t^{\sigma p}$,  $\sigma>1$. For $q> \sigma'(p-1)$, let $f\in L^A(v,\Omega)$ and $g\in L^B(v,\Omega)$, where $A(t)=t^{\sigma'(p-1)}\log(e+t)^q$ and $B(t)=t^{p\sigma'}\log(e+t)^{qp'}$.    Then, given any degenerate  weak subsolution (supersolution) $u\in QH_0^{1,p}(v,\Omega)$ of the Dirichlet problem~\eqref{dirichlet-problem}, 
 \begin{align}\label{infinity-bound}
 \|u^+~(u^-)\|_{L^\infty(v,\Omega)} \leq C(A,B,N,\vect,p, S_N,\Omega,v)\left(\|f\|_{L^A(v,\Omega)}+\|g\|_{L^B(v,\Omega)}\right),
 \end{align}
 where the constant is independent of $u$, $f$, and $g$. In particular, if $u$ is a degenerate weak solution,  $u\in L^\infty(v,\Omega)$ with this bound.
\end{theorem}

\begin{remark}
    Theorem~\ref{thm:power-gain} was proved when $p=2$ and $\vecg=0$ in~\cite{MR4280269}.  As was shown there, in the classical case when $v=1$ and $Q=I$, the identity matrix, there are examples to show that the theorem is false if $q<\sigma'-1 =\frac{n}{2}-1$ ($\frac{n}{2}$  is the dual of the gain in the classical Sobolev inequality), but it is unknown if the theorem is true for $\sigma'-1<q<\sigma'$.  

    In the case $1<p<\infty$, a more general version of Theorem~\ref{thm:no-gain}, but again with $g=0$,  was proved using a somewhat different approach in~\cite{10.48550/arxiv.2210.12441}.
\end{remark}

\medskip

Next, we consider the gain in the scale of Orlicz spaces.  In the case of power gain, as $\sigma$ decreases to $1$, we see that we need $f$ to be in an increasingly smaller Lebesgue space, so that it has more regularity.  For gain in Orlicz spaces, we see that $f$ can still be unbounded, but must be exponentially integrable.

\begin{theorem} \label{thm:orlicz-gain}
Given a bounded domain $\Omega\subset \mathbb{R}^n$ and $1<p<\infty$, suppose $Q$ and $v$ are such that 
$v\in L^1(\Omega)$, $|Q|_\op^{\frac{p}{2}} \leq kv$, and $\Omega$ supports the Sobolev inequality \eqref{sobolev} with $N(t)=t^p\log(e+t)^{\sigma}$,  $\sigma>p$. For $\alpha>\frac{p(p-1)}{\sigma}$ and $\beta >\frac{p}{\sigma-p}$, let $f\in L^A(v,\Omega)$, where $A(t)=\exp(t^\alpha)-1$, and $g\in L^B(v,\Omega)$ with $B(t) = \exp(t^\beta)-1$.  
Then, given any degenerate weak subsolution (supersolution) $u\in QH_0^{1,p}(v,\Omega)$ of the Dirichlet problem \eqref{dirichlet-problem}, 
 \begin{align}\label{orlicz-infinity-bound}
 \|u^+~(u^-)\|_{L^\infty(v,\Omega)} \leq C(A,B,N,\vect,p, S_N,\Omega,v)\left(\|f\|_{L^A(v,\Omega)}+\|g\|_{L^B(v,\Omega)}\right),
 \end{align}
 where the constant is independent of  $u$, $f$, and $g$.  In particular, if $u$ is a degenerate weak solution,  $u\in L^\infty(v,\Omega)$ with this bound.
\end{theorem}

\begin{remark}
    While we have chosen to prove Theorem~\ref{thm:orlicz-gain} for Orlicz functions of the form $N(t)=t^p\log(e+t)^{\sigma}$, the argument can be adapted to work with other Young functions, for example, the functions $N(t)\approx \exp\big( [ \log(t^2)^{\frac{1}{m}}+1]^m\big)$, $m>1$,  which are  related to the Sobolev inequalities recently considered in~\cite{korobenko2024}.  See Section~\ref{section:geometric} below for details.
\end{remark}

\smallskip

If $f=0$, the condition on $g$ in Theorem~\ref{thm:orlicz-gain} is the best possible yielded by our proof.  However, the condition on $f$ can be improved, but requires linking the values of $\alpha$ and $\beta$, which seems unnatural.  We will discuss this further after the proof of Theorem~\ref{thm:orlicz-gain}.   By setting $g=0$ we can modify the proof  to get the following result, which has weaker conditions on both $\sigma$ and $\alpha$.  We conjecture that Theorem~\ref{thm:orlicz-gain} can be improved to give this bound on $\alpha$, but we have been unable to prove it.

\begin{theorem}\label{thm:orlicz-gain-g=0}
Given a bounded domain $\Omega\subset \mathbb{R}^n$ and $1<p<\infty$, suppose $Q$ and $v$ are such that 
$v\in L^1(\Omega)$, $|Q|_\op^{\frac{p}{2}} \leq kv$, and $\Omega$ supports the Sobolev inequality \eqref{sobolev} with $N(t)=t^p\log(e+t)^{\sigma}$,  $\sigma>p-1$. 
Given $\alpha>\frac{p-1}{\sigma-p+1}$, let $f\in L^A(v,\Omega)$, where $A(t)=\exp(t^\alpha)-1$.    Then, given any  degenerate weak subsolution (supersolution) $u\in QH_0^{1,p}(v,\Omega)$ of the Dirichlet problem \eqref{dirichlet-problem} with $g=0$, 
 \begin{align}\label{orlicz-infinity-bound}
 \|u^+~(u^-)\|_{L^\infty(v,\Omega)} \leq C(A,B,N,\vect,p, S_N,\Omega,v)\left(\|f\|_{L^A(v,\Omega)}\right),
 \end{align}
 where the constant is independent of  $u$ and $f$.  In particular, if $u$ is a degenerate weak solution,  $u\in L^\infty(v,\Omega)$ with this bound.
\end{theorem}

\medskip

We now consider the case of a Sobolev inequality with no gain.  In this case,
while solutions exist (see, for instance,~\cite{korobenko2021,tubitak1}), we can no longer prove that they are bounded, even if we assume $f,g\in L^\infty(v,\Omega)$.  We conjecture that a counterexample exists and that boundedness fails in general in this setting. However, while solutions may not be bounded, we are able to prove that weak solutions are exponentially integrable.

\begin{theorem} \label{thm:no-gain}
Given a bounded domain $\Omega\subset \mathbb{R}^n$ and $1<p<\infty$, suppose $Q$ and $v$ are such that 
$v\in L^1(\Omega)$, $|Q|_\op^{\frac{p}{2}} \leq kv$, and $\Omega$ supports the Sobolev inequality \eqref{sobolev} with no gain, that is, with $N(t)=t^p$.  Let $f,g\in L^\infty(v,\Omega)$. Then, given any  degenerate weak subsolution (supersolution) $u\in QH_0^{1,p}(v,\Omega)$ of the Dirichlet problem \eqref{dirichlet-problem},
\[ \|u^+~(u^-) \|_{L^{H}(v,\Omega)} \leq C(H,\vect,p, S_N, \Omega, v)\left( \|f\|_{L^\infty(v,\Omega)}+\|g\|_{L^\infty(v,\Omega)}\right), \]
where $H(t) = t^r\exp(t)$ for any $r>0$, and  the constant is independent of $u$, $f$, and $g$.  In particular, if $u$ is a degenerate weak solution,  $u\in L^H(v,\Omega)$ with this bound.
\end{theorem}

\begin{remark}
    The assumption that $f\in L^\infty(v,\Omega)$ is natural, since in Theorem~\ref{thm:orlicz-gain-g=0}, as $\sigma$ decreases to $p-1$, $\alpha\rightarrow \infty$.
\end{remark}

\begin{remark}
    The exponential integrability of solutions,  in place of boundedness, has been considered previously.  See, for instance~\cite[Theorem~4.1]{MR369884}.
\end{remark}

\begin{remark}
    There are two open questions related to Theorem~\ref{thm:no-gain} which we think are interesting.  First, if we assume further regularity on $f$--e.g., $f$ is continuous--is it possible to show that $u$ is bounded?  Second, what if we assumed a Sobolev inequality with loss, e.g., an inequality of the form $\|\varphi\|_{L^q(v,\Omega)} \leq \|\grad \varphi\|_{QL^p(\Omega)}$, where $1\leq q<p$?  Poincar\'e inequalities with loss (in fact, with $q=1$) have been considered previously.  See, for example, Haj\l asz and Koskela~\cite{MR1336257}.
\end{remark}

Since a Sobolev inequality with any gain implies a Sobolev inequality with no gain (see Lemma~\ref{normcompare} below), an immediate corollary to Theorem~\ref{thm:no-gain} in the endpoint case of Theorem~\ref{thm:orlicz-gain} is that solutions are exponentially integrable.

\begin{cor} \label{cor:endpt-orliczgain}
    Given a bounded domain $\Omega\subset \mathbb{R}^n$ and $1<p<\infty$, suppose $Q$ and $v$ are such that 
$v\in L^1(\Omega)$, $|Q|_\op^{\frac{p}{2}} \leq kv$, and $\Omega$ supports the Sobolev inequality \eqref{sobolev} with $N(t)=t^p\log(e+t)^{p}$.  
 Let $f,g\in L^\infty(v,\Omega)$.    Then, given any degenerate  weak subsolution (supersolution) $u\in QH_0^{1,p}(v,\Omega)$ of the Dirichlet problem \eqref{dirichlet-problem}, 
 \begin{align*}
 \|u^+~(u^-)\|_{L^H(v,\Omega)} \leq C(H,\vect,p, S_N,\Omega,v)\left(\|f\|_{L^\infty(v,\Omega)}+\|g\|_{L^\infty(v,\Omega)}\right)
 \end{align*}
 where $H(t)=t^r\exp(t)$ for any $r>0$, and the constant is independent of $u$ and $f$.  In particular, if $u$ is a degenerate weak solution,  $u\in L^H(v,\Omega)$ with this bound.
\end{cor}

By modifying the proof of Theorem~\ref{thm:no-gain} we can also prove an endpoint result for Theorem~\ref{thm:power-gain}.  This complements the integrability results proved in~\cite{macdonald2023}.  

\begin{cor} \label{cor:endpt-powergain}
    Given a bounded domain $\Omega\subset \mathbb{R}^n$ and $1<p<\infty$, suppose $Q$ and $v$ are such that 
$v\in L^1(\Omega)$, $|Q|_\op^{\frac{p}{2}} \leq kv$, and $\Omega$ supports the Sobolev inequality \eqref{sobolev} with $N(t)=t^{\sigma p}$,  $\sigma>1$. Let $f\in L^A(v,\Omega)$ and $g\in L^B(v,\Omega)$, where $A(t)=t^{\sigma'(p-1)}\log(e+t)^{\sigma'(p-1)}$ and $B(t) = t^{p\sigma'}\log(e+t)^{p\sigma'}$.    Then, given any degenerate weak subsolution (supersolution) $u\in QH_0^{1,p}(v,\Omega)$ of the Dirichlet problem \eqref{dirichlet-problem}, 
 \begin{align*}
 \|u^+~(u^-)\|_{L^H(v,\Omega)} \leq C(A,B,H,\vect,p,S_N,\Omega,v)\left(\|f\|_{L^A(v,\Omega)}+\|g\|_{L^B(v,\Omega)}\right),
 \end{align*}
 where $H(t)=t^r\exp(t)$ for any $r>0$, and the constant is independent of $u$ and $f$.  In particular, if $u$ is a degenerate weak solution,  $u\in L^H(v,\Omega)$ with this bound.
\end{cor}

\begin{remark}
The proof of Theorem~\ref{thm:no-gain} is based on  an iteration inequality and in the proof of  Corollary~\ref{cor:endpt-powergain} we actually prove a stronger iteration inequality that implies the previous one.  Given this, it is tempting to conjecture that a stronger result holds, replacing $H$ by a larger Young function.  However, we have been unable to take advantage of this stronger inequality to prove such a result, and we leave this as an open problem.
\end{remark}

\medskip
Finally, we note that in all of our theorems, we are working over $\R^n$ with  Lebesgue measure.  However, all of our proofs go through in a more general  measure space with measure $\mu$, provided that $\mu(\Omega)<\infty$ and $\mu$ has the minimal regularity required for the theory of Orlicz spaces to still hold.  Details are left to the interested reader.

\medskip

The remainder of this paper is organized as follows.  In Section~\ref{section:prelim} we give some preliminary results about Orlicz spaces and degenerate weak solutions of the Dirichlet problem.  In Section~\ref{section:main-proof}, we prove Theorems~\ref{thm:power-gain} and~\ref{thm:orlicz-gain}.  We have organized the proofs to show the  features they have in common, which makes it easier to adapt the argument to other settings.  In Section~\ref{section:nogain-proof} we prove Theorem~\ref{thm:no-gain} and Corollary~\ref{cor:endpt-powergain}.  Finally, in Section~\ref{section:geometric} we give some applications of our results to specific geometric settings considered in~\cite{korobenko2024,Korobenko:2016ue}.

\section{Preliminaries}
\label{section:prelim}

In this section we gather some preliminary definitions and results.
We begin with some notation.  The constant $n$ will always denote the
dimension of the underlying space $\rn$.  By $C$, $c$, etc., we will
mean a constant that may change at each appearance, but
whose value depends only on the underlying parameters.  If we want to
specify this dependence, we will write, for instance, $C(n,p)$, etc.
If we write $A\lesssim B$, we mean that there exists a constant $c$
such that $A\leq cB$.  If $A\lesssim B$ and $B\lesssim A$, we write
$A\approx B$.

By $\lip_\loc(\Omega)$ we mean the collection of all locally Lipschitz functions defined on $\Omega$; by $\lip_0(\Omega)$ we mean functions in $\lip_\loc(\Omega)$ with compact support in $\Omega$.  

The weight $v$ will always be a
nonnegative, measurable function such that $v\in L^1(\Omega)$.  Given
a set $E\subset \Omega$, $v(E)=\int_E v(x)\,dx$.
Given a weight $v$ and $1\leq p < \infty$, $L^p(v,\Omega)$ is the collection of all those measurable functions $g:\Omega \rightarrow \mathbb{R}$ for which
$$ \|g\|_{L^p(v,\Omega)} = \displaystyle\left(\int_\Omega |g(x)|^p~v(x)dx\right)^{\frac{1}{p}} <\infty.$$
When $p=\infty$, $\|f\|_{L^\infty(v,\Omega)}$ is the essential supremum with respect to the measure $vdx$.
Note that $\lip_0(\Omega) \subset L^p(v,\Omega)$ for every $1\leq p \leq\infty$.

The matrix $Q$ will always be an $n\times n$, nonnegative definite, symmetric matrix-valued function defined on $\Omega$.  We assume that $Q$ is measurable, i.e., each of its coefficient functions is a measurable scalar function.  For fixed $p$, $1\leq p<\infty$,  We will assume that the largest eigenvalue of $Q$ is dominated by $v$:  more precisely, there exists a constant $k$ such that for a.e.~$x\in \Omega$, $|Q(x)|_\op^{\frac{p}{2}} \leq kv(x)$.   As a consequence, we have that $|\sqrt{Q}|_\op^p=|Q|_\op^{\frac{p}{2}} \in L^1(\Omega)$.  

$QL^p(\Omega)$ is the collection  of all
measurable, $\mathbb{R}^n$-valued functions ${\bf g}$ on $\Omega$ that satisfy
$$\|{\bf g}\|_{QL^p(\Omega)}
= \bigg(\int_\Omega |\sqrt{Q(x)}{\bf g}(x)|^p \,dx\bigg)^{\frac{1}{p}}<\infty.$$
With this norm and the equivalence $f\equiv g$ if and only if 
$\|f-g\|_{QL^p(\Omega)}=0,$
$QL^p(\Omega)$ becomes a Banach function space.  (See~\cite[Lemma~2.1]{CRR1}.)
Note that if $\varphi \in \lip_0(\Omega)$, then $\grad \varphi \in QL^p(\Omega)$.

\subsection*{Degenerate subunit vector fields}

Given a measurable function, ${\bf t}:\Omega\rightarrow\mathbb{R}^n$, the operator $\vect\cdot \nabla$ is called a vector field.  It is a degenerate subunit vector field if there is a constant $C(\vect)>0$ so that
\begin{align}\label{cond:subunit}
\|{\bf t}\cdot \nabla \varphi\|_{L^p(v,\Omega)} \leq C(\vect)\|\nabla \varphi\|_{QL^p(\Omega)} 
\end{align}
for every $\varphi\in Lip_0(\Omega)$.  This norm inequality holds if, for example, we assume the pointwise estimate
\begin{equation}\label{eqn:pointwisesubunit} 
\left| {\bf t}(x)\cdot \xi\right| \leq \frac{1}{\sqrt{v(x)}}\left| \sqrt{Q(x)}\xi\right|
\end{equation}
for a.e.~$x\in\Omega$ and every vector $\xi\in\mathbb{R}^n$.  Inequality~\eqref{eqn:pointwisesubunit} is adapted from~\cite{MR2204824}, where $v=1$. (See also~\cite{MR1488238,MR730094,MR}.)  We emphasize that condition \eqref{cond:subunit} is essential to control the divergence term on the right hand side of \eqref{dirichlet-problem} in our proofs below.

As a consequence of this definition, we have that our divergence term is a generalization of the  term $\Div(\vecg)$ that appears in a classical formulation of an elliptic PDE.  If $v=1$, then \eqref{eqn:pointwisesubunit} implies that $|\vect \cdot \xi| \leq |\xi|^2$, which in turn implies that $|\vect|\leq 1$.  Therefore, given a vector function $\vecg$, if we define $\vect = \vecg/|\vecg|$ and fix $g$ such that $|g|^{p-2}g=|\vecg|$, then
\[ \Div(|g|^{p-2}g\vect ) = \Div(\vecg).  \]

\subsection*{Orlicz spaces}
  Here we gather some
essential results about these spaces.  For complete information, we
refer to \cite{KR,RR}; for a succinct summary, see~\cite[Chapter~5]{CMP}.  
A Young function is a function $A : [0,\infty)\rightarrow
[0,\infty)$ that is continuous, convex, strictly increasing, 
$A(0)=0$, and $\frac{A(t)}{t}\rightarrow \infty$ as $t\rightarrow
\infty$.   Given a Young
function $A$, define $L^A(v,\Omega)$ to be the Banach function space of
measurable functions $h:\Omega\rightarrow \mathbb{R}$  with norm
$$\|h\|_{L^A(v,\Omega)} = \inf\bigg\{\lambda>0~:~\int_\Omega
A\left(\frac{|h(x)|}{\lambda}\right)~v(x)dx \leq 1\bigg\}<\infty.$$
Generally, the norm cannot be computed explicitly unless $A(t)=t^p$ and $L^A(v,\Omega)=L^p(v,\Omega)$.  However, the norm of a characteristic function can be computed.  The following lemma is an immediate consequence of the definition of the Luxembourg norm.

\begin{lemma} \label{lemma:norm-char}
Given a Young function $A$ and a set $E\subset \Omega$, 
\[  \|\chi_E\|_{L^A(v,\Omega)} = A^{-1}(v(E)^{-1})^{-1}.  \]
\end{lemma}
Given Young functions $A,\,B$ we can compare the associated norms by
appealing to a pointwise estimate.  We say that $A(t)\preceq B(t)$ if
there is a $t_0>0$ and a constant $c\geq 1$ depending only on $A,\,B$ so
that $A(t) \leq B(ct)$ for $t\geq t_0$.  Note that if $A(t) \lesssim B(t)$, then $A(t)\preceq B(t)$.  For a proof of the
following result, see~\cite[Theorem~13.3]{KR} or \cite[Section~5.3]{RR}.  

\begin{lemma}\label{normcompare} Given Young functions $A,\,B$, if
  $A\preceq B$, then there exists a constant $C=C(A,B,v,\Omega)$ such
  that for every $f\in L^B(v,\Omega)$,
$$\|f\|_{L^A(v,\Omega)} \leq C\|f\|_{L^B(v,\Omega)}.$$
\end{lemma}

\begin{remark}
    The fact that we can implicitly assume $t_0>0$ in Lemma~\ref{normcompare} is a consequence of our assumption that $\Omega$ is bounded and $v\in L^1(\Omega)$.  
\end{remark}

Given a Young function $A$, define the conjugate Orlicz function,
$\bar{A}$, by
\[ \bar{A}(t) = \sup\{ st-A(s) : s> 0 \}. \]
The pair $A,\,\bar{A}$ satisfy a  version of H\"older's
inequality in the scale of Orlicz spaces.  If $f\in L^A(v,\Omega)$ and
$g\in L^{\bar{A}}(v,\Omega)$, then $fg\in L^1(v,\Omega)$ and
\begin{equation} \label{holders}
\int_\Omega |f(x)g(x)|v(x)\,dx \leq 2\|f\|_{L^A(v,\Omega)}\|g\|_{L^{\bar{A}}(v,\Omega)}.
\end{equation}
More generally, we have that if $A,\,B,\,C$ are three Young functions such that
\[  B^{-1}(t)C^{-1}(t) \preceq A^{-1}(t), \]
then there exists a constant $K=K(A,B,C,v,\Omega)$ such that 
\begin{equation} \label{eqn:gen-holder}
\|fg\|_{L^A(v,\Omega)} \leq K\|f\|_{L^B(v,\Omega)}\|g\|_{L^C(v,\Omega)}. 
\end{equation}
For a proof of \eqref{holders} and \eqref{eqn:gen-holder}, see~\cite[Section~3.3]{RR}.
\medskip

In our main results we consider several kinds of Young functions:
power functions,
\[ \Phi(t) = t^r, \qquad 1\leq r < \infty; \]
power functions with a so-called ``log-bump",
\begin{equation} \label{eqn:log-bump}
  \Psi(t) = t^r\log(e+t)^q, \qquad 1\leq r,\, q < \infty;
\end{equation}
and exponential functions,
\[ \Theta(t) = \exp(t^r) -1, \qquad 1\leq r < \infty, \]
and
\[ \Xi(t) = t^r\exp(t), \qquad r > 0.  \]
Formulas for the inverse and conjugate functions
of $\Phi$ and $\Theta$ are immediate; for $\Psi$, see, for instance~\cite[Chapter~5]{CMP}.   We have that
\begin{gather*}
    \Phi^{-1}(t) = t^{\frac{1}{r}}, 
    \qquad \bar{\Phi}(t) \approx t^{r'}, \\
    \Psi^{-1}(t) \approx \frac{t^{\frac{1}{r}}}{\log(e+t)^{\frac{q}{r}}},
    \qquad \bar{\Psi} \approx \frac{t^{r'}}{\log(e+t)^{q(r'-1)}},\\
    \Theta^{-1}(t) \approx \log(1+t)^{\frac{1}{r}}, 
    \qquad \bar{\Theta}(t) \approx t\log(1+t)^{\frac{1}{r}}. 
\end{gather*}
In each case, the implicit constants depend on $r,\,q$.


\subsection*{Degenerate Sobolev spaces and weak solutions}
We now give a precise definition of degenerate weak
(sub)solutions to the Dirichlet problem~\eqref{dirichlet-problem}.  We only sketch the relevant details;  
see~\cite{CW,CRR1,CRR2,GT,MR,MRW,R,SW2} for further information.  

Fix $1\leq p<\infty$.  Given $v$ and $Q$, the
solution space for the Dirichlet problem is the matrix weighted
Sobolev space $QH_0^{1,p}(v,\Omega)$,  which is the 
 completion  of the space
$\lip_0(\Omega)$  with respect to the norm
$$\|\psi\|_{QH_0^{1,p}(v,\Omega)} = \|\psi\|_{L^p(v,\Omega)} + \|\nabla \psi\|_{QL^p(\Omega)}.$$

Formally,
$QH_0^{1,p}(v,\Omega)$ is the set of equivalence classes of Cauchy
sequences of $\lip_0(\Omega)$ functions.  However, since the spaces $L^p(v,\Omega)$ and
$ QL^p(\Omega)$ are complete, to each equivalence class
$[\{\psi_j\}]$ in  $QH_0^{1,p}(v,\Omega)$ we can associate a unique pair
$(u, {\bf g}) \in L^p(v,\Omega)\times QL^p(\Omega)$.  The norm of the pair is the norm of the 
equivalence class: that is,
\begin{equation*}
  \|(u,{\bf g})\|_{QH_0^{1,p}(v,\Omega)}
  = \|u\|_{L^p(v,\Omega)} + \|{\bf g}\|_{QL^p(\Omega)}
  =\lim_{j\rightarrow\infty}\left(\|\psi_j\|_{L^p(v,\Omega)}
    + \|\nabla \psi_j\|_{QL^p(\Omega)}\right).
\end{equation*}
Conversely, given a pair $(u,{\bf g})$ we will say that it is in
$QH_0^{1,p}(v,\Omega)$ if there exists a sequence $\{\psi_j\}_j\subset
\lip_0(\Omega)$ such that $(\psi_j, \nabla \psi_j)$ converges to $(u,{\bf g})$
in $L^p(v,\Omega)\times QL^p(\Omega)$.

Hereafter, we will denote ${\bf g}$ by $\nabla u$ since $\bf g$ plays
the role of a weak gradient of $u$, and we will often write $u\in QH_0^{1,p}(v,\Omega)$ rather than explicitly identifying the pair $(u, {\bf g})$ or $(u,\grad u)$.  However, we want to stress that the function ${\bf g}$ may not be the weak gradient of $u$ in the sense of classical Sobolev
spaces.  For further details, see~\cite[Section~2]{CRR1}, \cite[p.~1877]{SW2}, or \cite[Section~2.1]{MR643158}.

Throughout this paper we will assume the existence of a Sobolev inequality in the scale of Orlicz spaces: Given $v$ and $Q$, there exists a Young function $N$ such that for all $\varphi \in \lip_0(\Omega)$,
\begin{equation} \label{eqn:orlicz-sobolev}
\|\varphi\|_{L^N(v,\Omega)} \leq S_N \|\sqrt{Q}\grad \varphi\|_{L^p(\Omega)}.  
\end{equation}
This inequality extends to functions in $QH^{1,p}_0(v,\Omega)$.  

\begin{lemma}\label{sobolev2}
Suppose that for all $\varphi \in \lip_0(\Omega)$, inequality \eqref{eqn:orlicz-sobolev} holds.  Then it also holds, with the same constant, for any $u\in QH_0^{1,p}(v,\Omega)$. 
\end{lemma}

\begin{proof}
    Fix $u\in QH_0^{1,p}(v,\Omega)$. Then there exists a sequence $\{\varphi_j\}_{j=1}^\infty$ that converges to $u$ in $QH_0^{1,p}(v,\Omega)$.  Since this implies that $\varphi_j$ converges to $u$ in $L^p(v,\Omega)$,  by passing to a subsequence we may also assume that $\varphi_j$ converges to $u$ pointwise $v$-a.e.  Therefore, by Fatou's lemma in the scale of Banach function spaces~\cite[Chapter~1, Lemma~1.5]{MR928802},
    \[ \|u\|_{L^N(v,\Omega)} 
    \leq \liminf_{j\rightarrow \infty} \|\varphi_j\|_{L^N(v,\Omega)} 
    \leq S_N  \liminf_{j\rightarrow \infty}\|\grad \varphi_j\|_{QL^p(\Omega)}
    = S_N\|\grad u\|_{QL^p(\Omega)}. 
\]
\end{proof}

For our proof below we need a truncation property of $QH_0^{1,p}(v,\Omega)$ that is the analog of the fact that if $\varphi \in \lip_0(\Omega)$, then $\max(\varphi,0)\in \lip_0(\Omega)$.  This result was proved in~\cite[Lemma~2.14]{MR4280269} when $p=2$, and in~\cite[Lemma~2.12]{10.48550/arxiv.2210.12441} using the same ideas for $1<p<\infty$.  We note that it is for the proof of this result that we need to assume that $v\in L^1(\Omega)$ and that $|Q|_\op^{\frac{p}{2}} \leq kv$.  

\begin{lemma} \label{lemma:trunc}
Given $(u,\grad u)\in QH_0^{1,p}(v,\Omega)$, for every $r>0$ let $S(r)=\{ x\in \Omega : u(x)>r)$. Then $((u-r)_+, \chi_{S(r)}\grad u) \in QH_0^{1,p}(v,\Omega)$.  
\end{lemma}
\medskip

We can now define the degenerate weak (sub)solution to the Dirichlet problem. 

\begin{defn}\label{weaksol}
  Given $1<p<\infty$, a pair $(u,\nabla u)\in QH_0^{1,p}(v,\Omega)$ is 
  a degenerate weak solution of the Dirichlet problem \eqref{dirichlet-problem}
  if
\begin{multline*}
\int_\Omega  |\sqrt{Q}(x)\nabla u(x)|^{p-2}\nabla\varphi(x) \cdot Q(x)\nabla u (x) \,dx \\
  = \int_\Omega \big(f(x)|f(x)|^{p-2} \varphi(x)+  |g(x)|^{p-2}g(x) \vect \cdot \grad\varphi(x)\big)\,v(x)dx
\end{multline*}  
for every  $\varphi\in \lip_0(\Omega)$.  It is a
 weak subsolution (supersolution) if this holds with the equality ``$=$" replaced by the inequality ``$\leq$" (``$\geq$") for any nonnegative $\varphi\in \lip_0(\Omega)$.

\end{defn}

\begin{remark} \label{rem:QH-test}
Note that if $(h,\nabla h)\in QH_0^{1,p}(v,\Omega)$ with $h(x)\geq 0$ $v$-a.e., then by a
standard limiting argument we may use $h$ as our test function in
Definition~\ref{weaksol}.
\end{remark}

\begin{remark}
    If $(u,\nabla u)$ is a weak supersolution of \eqref{dirichlet-problem}, then it is immediate that $-(u,\nabla u)$ is a weak subsolution of \eqref{dirichlet-problem} with $f$ and $g$ replaced by $-f$ and $-g$.  Further, if $(u,\nabla u)$ is a weak solution, it is both a subsolution and supersolution. 
\end{remark}

Finally, given any extended real valued function $u$, let $u^+=\max(u,0)$; and let $u^-=\max(-u,0)$. Note that for any $u$, $u=u^+-u^-$.  



\section{Boundness: Proof of results}
\label{section:main-proof}

In this section we prove Theorems~\ref{thm:power-gain} and~\ref{thm:orlicz-gain}.  The proofs are based on the classical iteration method of de Giorgi~\cite{DeGiorgi}.  Our version, which builds on the approach used in~\cite{MR4280269}, is modeled on the adaptation of de Giorgi iteration given by Korobenko, {\em et al.}~\cite{Korobenko:2016ue}.  Their proof in turn  used ideas from Christ~\cite{MR1912731}.  The proofs of both  theorems are very similar in structure, though the specific argument changes depending on the Sobolev inequality we assume. We will first give the parts of the proof that are common to both; subsequently,  we will give the specific arguments required for each result.

\subsection*{Foundation of the iteration argument}
We divide our argument into three lemmas. In the first, we prove an inequality which is the foundation of the iteration argument.  Note that in each proof below we will determine the Young function $D$ which appears in this lemma.

\begin{lemma}\label{Foundation1}
Let $1<p<\infty$ and suppose that the Sobolev inequality~\eqref{sobolev} holds for a Young function $N$.  Set $M(t) = \overline{N}(t^{p-1})$.  Let $u$ be a  degenerate weak subsolution of \eqref{dirichlet-problem} with data $f\in L^A(v,\Omega)$ and $g\in L^B(v,\Omega)$ for Young functions $A,\,B$.  If $D$ is a Young function satisfying 
\begin{equation} \label{eqn:D-assumption} 
A^{-1}(t)D^{-1}(t) \lesssim M^{-1}(t), \qquad  B^{-1}(t)D^{-1}(t) \lesssim  t^\frac{1}{p},
\end{equation}
then for all  $s> r>0$, 
\begin{equation}\label{eqn:next-step}
(s-r) N^{-1}(v(S(s))^{-1})^{-1}
\leq L
\left(\|f\|_{L^A(v,\Omega)}+\|g\|_{L^B(v,\Omega)}\right) D^{-1}(v(S(r))^{-1})^{-1},
\end{equation}
where  $L=L(A,B,D,N,\vect,p, S_N)$ and $S(r) = \{x\in\Omega : u(x)>r\}$.
\end{lemma}

\begin{proof}
Fix $r>0$, let $\phi_r=(u-r)_+$, and define $S(r)=\{ x\in \Omega : u(x)>r\}$.  By Lemma~\ref{lemma:trunc},  $(\phi_r, \grad\phi_r) = ((u-r)_+, \chi_{S(r)}\grad u)\in QH_0^{1,p}(v,\Omega)$; thus, by Remark~\ref{rem:QH-test}, it can be used as a nonnegative test function for the degenerate weak subsolution~$u$. By the Sobolev inequality (Lemma~\ref{sobolev2}) and the definition of a subsolution, we get that 
\begin{align*}
    \|\grad \phi_r\|_{QL^p(\Omega)}^p 
    & =   \int_{S(r)} |\sqrt{Q}\grad \phi_r|^p\,dx \\
    & =  \int_{S(r)} |\sqrt{Q}\grad \phi_r|^{p-2} Q\grad \phi_r \cdot \grad \phi_r \,dx \\
    &   =  \int_{S(r)} |\sqrt{Q}\grad u|^{p-2} Q\grad u \cdot \grad \phi_r \,dx \\
    &\leq  \int_{S(r)} |f|^{p-1}\phi_r \,vdx + \int_{S(r)} |g|^{p-1}|\vect\cdot \grad \phi_r|\,vdx. 
\end{align*}
We estimate each integral in the last line separately.  For the first, we apply H\"older's inequality in the scale of Orlicz spaces~\eqref{holders} with Young functions $N$ and $\bar{N}$, and again apply the Sobolev inequality and rescaling, using the fact that $M(t)=\bar{N}(t^{p-1})$, to get
\begin{multline*} \int_{S(r)} |f|^{p-1}\phi_r \,vdx  \\
\leq
2\||f|^{p-1}\chi_{S(r)}\|_{L^{\bar{N}}(v,\Omega)}\|\phi_r\|_{L^N(v,\Omega)}
\leq
2S_N\|f\chi_{S(r)}\|_{L^{M}(v,\Omega)}^{p-1}\|\nabla \phi_r\|_{QL^p(\Omega)}.
\end{multline*}
To estimate the second integral, we use 
H\"older's inequality with exponents $p$ and $p'$, and the subuniticity condition~\eqref{cond:subunit} to get
\begin{multline*}  
\int_{S(r)} |g|^{p-1}|\vect\cdot \grad \phi_r|\,vdx \\
    \leq  
    \||g|^{p-1}\chi_{S(r)}\|_{L^{p'}(v,\Omega)}\|\vect\cdot \grad \phi_r\|_{L^p(v,\Omega)}
    \leq
  C(\vect)\|g\chi_{S(r)}\|_{L^p(v,\Omega)}^{p-1}\|\nabla \phi_r\|_{QL^p(\Omega)}.
\end{multline*}
If we combine the above inequalities, rearrange terms, and take the $p-1$ root, we get
\begin{equation*}   
\|\grad \phi_r\|_{QL^p(v,\Omega)} 
 \leq 
L(\vect,p,S_N)\big(\|f\chi_{S(r)}\|_{L^{M}(v,\Omega)}+\|g\chi_{S(r)}\|_{L^p(v,\Omega)}\big).  
\end{equation*}
Because \eqref{eqn:D-assumption} holds, we can apply the Sobolev inequality and the generalized H\"older's inequality~ \eqref{eqn:gen-holder} to the integrals of $f$ and $g$ to get 
\begin{multline*}   
\|\phi_r\|_{L^N(v,\Omega)} 
\leq S_N\|\grad \phi_r\|_{QL^p(v,\Omega)} \\
\leq L(A,B,D,N,\vect,p,S_N)\big(\|f\|_{L^A(v,\Omega)}+\|g\|_{L^B(v,\Omega)}\big)\|\chi_{S(r)}\|_{L^D(v,\Omega)}.
\end{multline*}

To complete the proof, fix $s>r$. Then $S(s)\subset S(r)$ and for all $x\in  S(s)$,  $\phi_r(x)>s-r$.  Hence,
\begin{equation*}
    \|\phi_r\|_{L^N(v,\Omega)} \geq \|\phi_r\chi_{S(s)}\|_{L^N(v,\Omega)} 
    \geq (s-r)\|\chi_{S(s)}\|_{L^N(v,\Omega)}.  
\end{equation*}
If we combine this with the above inequality, we get
\[
(s-r)\|\chi_{S(s)}\|_{L^N(v,\Omega)}
\leq L
\left(\|f\|_{L^A(v,\Omega)}+\|g\|_{L^B(v,\Omega)}\right)\|\chi_{S(r)}\|_{L^D(v,\Omega)}.
\]
By Lemma~\ref{lemma:norm-char} we can rewrite this as 
\begin{equation} 
(s-r) N^{-1}(v(S(s))^{-1})^{-1}
\leq L
\left(\|f\|_{L^A(v,\Omega)}+\|g\|_{L^B(v,\Omega)}\right) D^{-1}(v(S(r))^{-1})^{-1},
\end{equation}
which is~ \eqref{eqn:next-step}.
\end{proof}
\medskip

Our second lemma is a norm regularity result for solutions of the Dirichlet problem.   The proof uses the weak solution itself as a test function and leverages the Sobolev inequality.  This argument is classical:  see, for instance, \cite[Theorem 8.1]{GT}.   In the study of degenerate PDEs, this technique has been used to prove Caccipolli-type estimates for a maximum principle with a different scale of Orlicz Sobolev-type inequalities in~\cite[Chapter 5]{Korobenko:2016ue}; it was used in~\cite{MR4280269}, which considered the case $p=2$ of Theorem~\ref{thm:power-gain}; and it was used in~\cite[Lemma 3.1]{CRR1} to prove weighted $L^p$-norm regularity estimates in the study of Neumann problems.

\begin{lemma} \label{lemma-m0}
 Given Young functions $N$, $A$, and $B$, suppose $N$ satisfies \eqref{sobolev}, and $A$ and $B$ satisfy
 \begin{equation} \label{eqn:mo1} 
 A(t)\succeq M(t)=\bar{N}(t^{p-1}), \qquad B(t)\succeq t^p. 
 \end{equation}
 If $u$ is a degenerate weak subsolution of \eqref{dirichlet-problem} with $f\in L^A(v,\Omega)$ and $g\in L^B(v,\Omega)$,  then $u$ satisfies  
\begin{align}\label{lemma-pre-m0}\|u^+\|_{L^N(v,\Omega)} \leq C(A,B,N, \vect,p,S_N)\left(\|f\|_{L^A(v,\Omega)}+\|g\|_{L^B(v,\Omega)}\right).
\end{align}
\end{lemma}

\begin{remark}
    Note that the assumption~\eqref{eqn:D-assumption} in Lemma~\ref{Foundation1} implies \eqref{eqn:mo1}.
\end{remark}

\begin{proof} 
Since $u$ is a degenerate weak subsolution, by Remark~\ref{rem:QH-test}, we can use $(u^+,\grad u^+) = (u^+,\chi_{\{u>0\}}\nabla u)$ as a test function in the definition.  Therefore, we can argue as we did in the proof of Lemma~\ref{Foundation1}, but using Lemma~\ref{normcompare} and \eqref{eqn:mo1} instead of the generalized H\"older's inequality, to get
\begin{align*}
 \|\nabla u\|_{QL^p(S(0))}^p 
&\leq  \int_\Omega u^+|f|^{p-1}\,vdx +  \int_{S(0)} |g|^{p-1}|\vect\cdot \grad u|\,vdx\\
&\leq 2\|u^+\|_{L^N(v,\Omega)}\||f|^{p-1}\|_{L^{\overline{N}}(v,\Omega)} 
+ \||g|^{p-1}\|_{L^{p'}(v,\Omega)}\|\vect \cdot \grad u\|_{L^p(S(0))}\\
&=2S_N\|\grad u\|_{QL^p(S(0))}\|f\|_{L^M(v,\Omega)}^{p-1} + C(\vect)\|g\|_{L^p(v,\Omega)}^{p-1}\|\nabla u\|_{QL^p(S(0))}\\
&\leq C(A,B,N,\vect,p,S_N)\left(\|f\|_{L^A(v,\Omega)}^{p-1}+\|g\|_{L^B(v,\Omega)}^{p-1}\right)\|\nabla u\|_{QL^p(S(0))}. 
\end{align*}
If we divide by $\|\nabla u\|_{QL^p(S(0))}$, take the $p-1$ root, and then apply the Sobolev inequality (Lemma~\ref{sobolev2}), we get~\eqref{lemma-pre-m0}. 
\end{proof}

Our third lemma proves the actual iteration inequality we will use in our proofs below.

\begin{lemma} \label{lemma:iteration}
Given the hypotheses of Lemma~\ref{Foundation1}, for fixed $\epsilon>0$, $\tau_0>1$, and for $k \in \N$ define 
\[ D_k = \tau_0\left(\|f\|_{L^A(v,\Omega)}+\|g\|_{L^B(v,\Omega)}\right)\bigg(1-\frac{1}{(k+1)^\epsilon}\bigg),  \]
  Define $D_0=D_1/2$ and for all $k\geq 0$ let $\mu_k=v(S(D_k))$.  Then for $k\geq 1$ we have that
\begin{equation} \label{eqn:foundation}
 N^{-1}(\mu_{k+1}^{-1})^{-1}
\leq L\bigg(\frac{(k+2)^{1+{\epsilon}}}{\epsilon \tau_0}\bigg)
 D^{-1}(\mu_k^{-1})^{-1},
\end{equation}
and 
\begin{equation} \label{eqn:foundation-0term}
 N^{-1}(\mu_{1}^{-1})^{-1} \leq \frac{4L}{\tau_0(1-2^{-\epsilon})} D^{-1}(\mu_0^{-1})^{-1},
\end{equation}
where $L$ is as in~\eqref{eqn:next-step}.
Moreover, we can choose $\tau_0$ sufficiently large that $\mu_0\leq e^{-2}$.  
\end{lemma}

\begin{proof}
By the mean value theorem, we have that for each $k\geq 1$,
\[ D_{k+1} - D_k \geq \tau_0 \left(\|f\|_{L^A(v,\Omega)}+\|g\|_{L^B(v,\Omega)}\right)\frac{\epsilon }{(k+2)^{1+\epsilon}}.
\]
This sequence is decreasing, so if we let $s=D_{k+1}$ and $r=D_k$ and combine this estimate with~\eqref{eqn:next-step}, we get~\eqref{eqn:foundation}.  When $k=0$, the same argument, using that $s-r=D_1/2$,  gives inequality~\eqref{eqn:foundation-0term}.

To see that  for $\tau_0$  sufficiently large  $\mu_0\leq e^{-2}$,  note that if $x\in S(D_0)$, then $2u(x)/D_1 > 1$ $v$-a.e.  Hence, by Lemma~\ref{lemma-m0} and  H\"older's inequality~\eqref{holders} with Young function $N$, 
\begin{align*}
\mu_0 
& \leq 2D_1^{-1}\int_{S(D_0)} u\,vdx \\
& \leq  2D_1^{-1}\|\chi_\Omega\|_{L^{\bar{N}}(v,\Omega)}\|u^+\|_{L^N(v,\Omega)} \\
& \leq D_1^{-1}C(A,B,N,\vect,p,S_N,\Omega,v)\left(\|f\|_{L^A(v,\Omega)}+\|g\|_{L^B(v,\Omega)}\right).
\end{align*}
If we now insert the definition of $D_1$, we see that
\[\mu_0 \leq \frac{C(A,B,N, \vect,p,S_N, \Omega,v)}{\tau_0(1-2^{-\epsilon})}\]
and so we can choose $\tau_0$ so large that $\mu_0\leq e^{-2}$.
\end{proof}

\medskip

We now turn to the proofs of each boundedness theorem.  In each proof, we will show that for $\tau_0$ sufficiently large, $\mu_k\rightarrow 0$ as $k\rightarrow \infty$; by the continuity of the integral we get that $v(S(r_0))=0$, which completes the proof.

\subsection*{Proof of Theorem~\ref{thm:power-gain}}
Recall that $N(t)=t^{\sigma p}$, $\sigma>1$, $A(t)=t^{\sigma'(p-1)}\log(e+t)^q$, $B(t)=t^{p\sigma'}\log(e+t)^{qp'}$ where $q>\sigma'(p-1)$.
To apply Lemma~\ref{lemma:iteration} we need inequality~\eqref{eqn:D-assumption} to hold.  In fact we can find a Young function $D$ such that
\[ A^{-1}(t) D^{-1}(t) \approx M^{-1}(t) \hbox{ and } B^{-1}(t)D^{-1}(t) \approx t^{\frac{1}{p}},\]
where $ M(t) = \bar{N}(t^{p-1}) \approx t^{(p\sigma)'(p-1)}$.  
To see this, note that since 
\[ A^{-1}(t) \approx \frac{t^{\frac{1}{\sigma'(p-1)}}}{\log(e+t)^{\frac{q}{\sigma'(p-1)}}} 
\quad \hbox{ and } \quad 
B^{-1}(t)\approx \frac{t^\frac{1}{p\sigma'}}{\log(e+t)^\frac{q}{(p-1)\sigma'}},  \]
if we solve either approximate equality, we get
\[ D^{-1}(t) \approx t^{\frac{1}{(p\sigma)'(p-1)}- \frac{1}{\sigma'(p-1)}}\log(e+t)^{\frac{q}{\sigma'(p-1)}}
  = t^{\frac{1}{p\sigma}}\log(e+t)^{\frac{q}{\sigma'(p-1)}}.  \]

Given this, inequality~\eqref{eqn:foundation} holds and in this context becomes
\begin{equation} \label{eqn:iter-power}
 \mu_{k+1}^{\frac{1}{\sigma p}}  
\leq  L\bigg(\frac{(k+2)^{1+{\epsilon}}}{\epsilon \tau_0}\bigg) 
 \frac{\mu_{k}^{\frac{1}{\sigma p}}}{\log(\mu_k^{-1})^{\frac{q}{\sigma'(p-1)}}}.  
 \end{equation}
Now define  $m_k= \log(\mu_k^{-1})$, $k\geq 0$.  We will show   by induction, for all $k\geq 1$, $m_{k} \geq m_0 +k$.  Given this,  we have that  $m_k\rightarrow \infty$ as $k\rightarrow \infty$; this implies that 
\[ v\left(S\left(\tau_0 \left(\|f\|_{L^A(v,\Omega)}+\|g\|_{L^B(v,\Omega)}\right)\right)\right)=0, \]
which completes the proof for degenerate weak subsolutions.

We first show that $m_1 \geq m_0+1$. In this case, inequality~\eqref{eqn:foundation-0term} becomes
\begin{equation*}
\mu_1^{\frac{1}{\sigma p}} \leq  \frac{4L}{\tau_0(1-2^{-\epsilon})} \frac{\mu_0^{\frac{1}{\sigma p}}}{\log(\mu_0^{-1})^{\frac{q}{\sigma'(p-1)}}}.
\end{equation*}
If we raise both sides to the power $\sigma p$, take the reciprocal, and  take the logarithm of both sides, then by the  definition of $D_1$ we have that
\begin{align*}
m_1 \geq \sigma p\log\left(\frac{\tau_0(1-2^{-\epsilon})}{4L} \right) + m_0 
+ \Big(\frac{q\sigma p}{\sigma'(p-1)}\Big)\log(m_0).
\end{align*}
Again by Lemma~\ref{lemma:iteration}, we can choose $\tau_0$ sufficiently large so that $\log(m_0)\geq \log(2)>0$, so  the last term on the right is nonnegative.  Therefore,  by taking  $\tau_0$ sufficiently large (depending on our choice of $\epsilon$ below) we have  that  $m_1 \geq m_0+1$. 

We now prove the induction step.  Suppose for some $k$, $m_k\geq m_0+k$.  Arguing as we did above, \eqref{eqn:foundation} becomes
\[ m_{k+1}
\geq \sigma p\log\Big(\frac{\epsilon \tau_0}{L}\Big) -\sigma p(1+\epsilon)\log(k+2) +m_k 
+\frac{q\sigma p}{\sigma'(p-1)} \log(m_k). \]
Fix $\epsilon>0$ such that 
\[ 1+\epsilon = \frac{q}{\sigma'(p-1)};\]
this is possible by our assumptions that $q>\sigma'(p-1)$.  Then we can combine terms to get
\[ m_{k+1}
\geq \sigma p\log\Big(\frac{\epsilon \tau_0}{L}\Big)  +m_k 
+\frac{q\sigma p}{\sigma'(p-1)} \log\left(\frac{m_k}{k+2}\right). \]
By our induction hypothesis and since $m_0\geq 2$, the last term is positive. Therefore, we can choose $\tau_0$ sufficiently large independently of $k$ so that 
\[ m_{k+1} \geq 1 + m_0+k = m_{0}+(k+1).  \]
To finish the proof, given a degenerate weak supersolution $w$, consider the subsolution $u=-w$ of \eqref{dirichlet-problem} with data $-f|f|^{p-2}+v^{-1}\Div(v|g|^{p-2}g{\bf t})$.  Then by our argument above we have that 
\[\|w^{-}\|_{L^\infty(v,\Omega)}= \| u^+ \|_{L^\infty(v,\Omega)} \leq C\left(\|f\|_{L^A(v,\Omega)}+\|g\|_{L^B(v,\Omega)}\right).\]
Finally, if $u$ is a solution, it is both a subsolution and supersolution, and the norm estimate follows from the fact that $u=u^+ - u^-$.  This completes the proof.

\subsection*{Proof of Theorem~\ref{thm:orlicz-gain}}
The proof is similar in outline to the proof of Theorem~\ref{thm:power-gain}, but many small significant changes are required.  Therefore, despite some possible redundancy, we give  the details of the iteration argument.  

Recall that  $N(t)=t^p\log(e+t)^\sigma$, $\sigma>p$,  $A(t)=\exp(t^\alpha)-1$, $B(t)=\exp(t^\beta)-1$, where 
$\alpha>\frac{p(p-1)}{\sigma}$ and $\beta>\frac{p}{\sigma-p}$. To apply Lemma~\ref{lemma:iteration} we need inequality~\eqref{eqn:D-assumption} to hold.  In fact, we will find $D$ such that
\[ B^{-1}(t)D^{-1}(t)\approx t^\frac{1}{p} 
\quad \hbox{ and } \quad 
A^{-1}(t) D^{-1}(t) \lesssim  M^{-1}(t). \]
We have that 
\[ M(t) = \bar{N}(t^{p-1}) \approx \frac{t^{p'(p-1)}}{\log(e+t)^{\sigma(p'-1)}}
= \frac{t^p}{\log(e+t)^{\frac{\sigma}{p-1}}}, \]
and for $t\geq 1$,
\[  A^{-1}(t)  \approx \log(e+t)^{\frac{1}{\alpha}},
\quad~B^{-1}(t)\approx \log(e+t)^\frac{1}{\beta}, \text{ and } 
 M^{-1}(t) \approx t^{\frac{1}{p}} \log(e+t)^{\frac{\sigma}{p(p-1)}}.  
\]
If we solve the first approximate equality for $D^{-1}$, we get

Hence, 
\[  D^{-1}(t) \approx \frac{t^\frac{1}{p}}{B^{-1}(t)} 
\approx 
t^\frac{1}{p}\log(e+t)^{-\frac{1}{\beta}}.  
 \]
To justify the second approximate inequality, note that
\[ \frac{M^{-1}(t)}{A^{-1}(t)} \approx t^\frac{1}{p}\log(e+t)^{\frac{\sigma}{p(p-1)}-\frac{1}{\alpha}},
\]
and the exponent on the log term on the righthand side is positive by our assumption on $\alpha$.  

Given this, inequality~\eqref{eqn:D-assumption} holds, and in this context \eqref{eqn:next-step} becomes
\[ \mu_{k+1}^{\frac{1}{p}} \log(\mu_{k+1}^{-1})^{\frac{\sigma}{p}}
\leq L\left(\frac{(k+2)^{1+\epsilon}}{\epsilon \tau_0}\right)
\mu_k^{\frac{1}{p}}\log(\mu_k^{-1})^{\frac{1}{\beta}}.
\]
Now define  $m_k= \log(\mu_k^{-1})$, $k\geq 0$.  We will show   by induction, for all $k\geq 1$, $m_{k} \geq m_0 +k$.   Then, arguing exactly as we did in the proof of Theorem~\ref{thm:power-gain}, we get the desired result.

\medskip

We first show that $m_1 \geq m_0+1$.  In this case, inequality~\eqref{eqn:foundation-0term} becomes
\[  \mu_1^{\frac{1}{p}}\log(\mu_1^{-1})^{\frac{\sigma}{p}}
\leq \frac{4L}{\tau_0 (1-2^{-\epsilon})}\mu_0^{\frac{1}{p}}\log(\mu_0^{-1})^{\frac{1}{\beta}}.
\]
If we raise this to the power $-p$ we get
\[ \mu_1^{-1} \log(\mu_1^{-1})^{-\sigma} 
\geq \left(\frac{\tau_0 (1-2^{-\epsilon})}{4L}\right)^p 
\mu_0^{-1} \log(\mu_0^{-1})^{-\frac{p}{\beta}}.  
\]
If we take the log of both sides and rearrange terms, using the fact that $m_0\leq m_1$, we get 
\[ m_1  \geq p\log\left(\frac{\tau_0(1-2^{-\epsilon})}{4L} \right) + m_0 
+ \Big(\sigma-\frac{p}{\beta}\Big)\log(m_0).  \]
By Lemma~\ref{lemma:iteration}, for $\tau_0$ sufficiently large, $\log(m_0)>0$, and so, by our assumption on $\beta$, the last term is nonnegative.   Therefore, for $\tau_0$ sufficiently large (depending on our choice of $\epsilon$ below), we have that $m_1\geq m_0+1$.

We now prove the induction step.  Suppose for some $k$, $m_k\geq m_0+k$.  Arguing as we did above, now using the fact that $m_{k+1}\geq m_k$, \eqref{eqn:foundation} becomes
\[ m_{k+1}
\geq p\log\Big(\frac{\epsilon \tau_0}{L}\Big) -p(1+\epsilon)\log(k+2) +m_k 
+\Big(\sigma-\frac{p}{\beta}\Big) \log(m_k). \]

Fix $\epsilon>0$ such that $\frac{\sigma}{p}-\frac{1}{\beta}=1+\epsilon$; this is possible by our assumption on $\beta.$ With this definition of $\epsilon$, we can combine the two log terms to get (since $m_{k+1}\geq 1$)
\[ m_{k+1} \geq p\log\Big(\frac{\epsilon \tau_0}{L}\Big)
+m_k + p(1+\epsilon)\log\Big(\frac{m_k}{k+2}\Big). \]

By our induction hypothesis and since  $m_0\geq 2$, the last term  is positive. Therefore, we can choose  $\tau_0$ sufficiently large, independent of $k$, so that
\[ m_{k+1} \geq 1 + m_0+k = m_{0}+(k+1).  \]
This completes the proof for subsolutions.  The proof for supersolutions and solutions is the same as in Theorem~\ref{thm:power-gain}.

\subsection*{Proof of Theorem~\ref{thm:orlicz-gain-g=0}}
The proof of this result is essentially the same as the proof of Theorem~\ref{thm:orlicz-gain}.  Since $g=0$,  the Young function $B$ does not play a role and we only need to find a Young function $D$ such that $A^{-1}(t)D^{-1}(t) \approx M^{-1}(t)$.  The argument above shows that we can take
\[ D^{-1}(t) \approx t^\frac{1}{p}\log(e+t)^{\frac{\sigma}{p(p-1)}-\frac{1}{\alpha}}.\]
We can then repeat the above argument using this definition of $D^{-1}$.  In it, we  need to pick our value of $\epsilon$ so that $\frac{\sigma}{p-1}-\frac{1}{\alpha}=1+\epsilon$.  This is possible because of our assumption that $\alpha > \frac{p-1}{\sigma-p+1}$, which in turn is where we use our assumption that $\sigma>p-1$.  

\medskip

\begin{remark}
    In the statement of Theorem~\ref{thm:orlicz-gain}, we chose a condition on $\alpha$ so that $\frac{\sigma}{p(p-1)}-\frac{1}{\alpha}>0$, which was sufficient to prove that $A^{-1}(t)D^{-1}(t)\lesssim M^{-1}(t)$.  However, this would hold if we assumed the weaker condition that 
    \begin{equation} \label{eqn:alpha-beta-link} 
    \frac{\sigma}{p(p-1)}-\frac{1}{\alpha} \geq - \frac{1}{\beta}.  
    \end{equation}
    If we combine this with the assumption that $\beta>\frac{p}{\sigma-p}$, we see that this implies that $\alpha > \frac{p-1}{\sigma-p+1}$, the condition in Theorem~\ref{thm:orlicz-gain-g=0}.  However, \eqref{eqn:alpha-beta-link} is a stronger assumption, and the value of $\alpha$ depends on our choice of $\beta$.  It is an open problem whether the values of $\alpha$ and $\beta$ can be decoupled to prove a sharper version of Theorem~\ref{thm:orlicz-gain}.
\end{remark}

\section{Exponential integrability:  Proof of Theorem~\ref{thm:no-gain} and Corollary~\ref{cor:endpt-powergain}}
\label{section:nogain-proof}

We first prove Theorem~\ref{thm:no-gain}.  We will prove this for subsolutions; the proof for supersolutions and solutions is gotten as in the previous proofs.
Recall that we assume $N(t)=t^p$ in the Sobolev inequality~\eqref{sobolev}.    Let $u$ be a  degenerate weak subsolution of \eqref{dirichlet-problem} with $f,g\in L^\infty(v,\Omega)$.  We now repeat the argument used to establish the common iteration formula~\eqref{eqn:next-step} at the beginning of Section~\ref{section:main-proof}  by testing $u$ against $\varphi_r$.  As in the proofs above, for $r>0$ set $S(r) = \{x\in\Omega~:~u(x)>r\}$ and set $\varphi_r(x) = (u(x)-r)_+$.   Then, since $\supp(\varphi_r)=S(r)$,
\begin{align*}
& \|\nabla \varphi_r\|_{QL^p(S(r))}^p \\
&\qquad \quad \leq \int_{S(r)}|f|^{p-1}\varphi_r\,vdx + \int_{S(r)}|g|^{p-1}|\vect\cdot \grad \varphi_r|\,vdx\\
&\qquad \quad \leq \|\varphi_r\|_{L^p(v,\Omega)}\|f^{p-1}\chi_{S(r)}\|_{L^{p'}(v,\Omega)}+ \|g^{p-1}\chi_{S(r)}\|_{L^{p'}(v,\Omega)}\|\vect \cdot \grad \varphi_r\|_{L^p(v,S(r))}\\
&\qquad \quad \leq \left(S_N\|f\chi_{S(r)}\|_{L^p(v,\Omega)}^{p-1} + C(\vect)\|g\chi_{S(r)}\|_{L^p(v,\Omega)}^{p-1}\right)\|\nabla \varphi_r\|_{QL^p(S(r))}\\
&\qquad \quad \leq C(\vect,p, S_N)\left(\|f\|_{L^\infty(v,\Omega)}+\|g\|_{L^\infty}\right)^{p-1}\|\chi_{S(r)}\|_{L^p(v,\Omega)}^{p-1}\|\nabla \varphi_r\|_{QL^p(S(r))}.
\end{align*}
If we divide by $\|\nabla \varphi_r\|_{QL^p(S(r))}$, take the $p-1$ root, and use the Sobolev inequality without gain, we get
\[ \|\varphi_r\|_{L^p(v,S(r))} \leq C(\vect,p, S_N)\left(\|f\|_{L^\infty(v,\Omega)}+\|g\|_{L^\infty(v,\Omega)}\right)\|\chi_{S(r)}\|_{L^p(v,\Omega)}.\]
Thus, we can argue as before, and for $s>r>0$ we find
\begin{align}\label{2-2-foundation}
(s-r)v(S(s))^{\frac{1}{p}} \leq C(\vect,p, S_N)\left(\|f\|_{L^\infty(v,\Omega)}+\|g\|_{L^\infty(v,\Omega)}\right)v(S(r))^{\frac{1}{p}}.
\end{align}

Suppose first that $0<\|f\|_{L^\infty(v,\Omega)}+\|g\|_{L^\infty(v,\Omega)}\leq 1$.  For $k\in \N$, define 
\[ r_k = C(\vect,p, S_N)ek=\alpha k. \]
If we set $r=r_k$ and $s=r_{k+1}$ and define $\mu_k = v(S(r_k))$, \eqref{2-2-foundation} implies that
\[\mu_{k+1} \leq e^{-p}\mu_k \leq \dots \leq e^{-pk}\mu_1.\]

By the definition of the Orlicz norm,  $u^+\in L^H(v,\Omega)$, where $H(t)=t^r\exp(t)$, $r>0$,  if and only if there exists $\lambda>0$ such that 
\begin{equation}\label{2-2-orliczinclusion} 
\int_\Omega H\left(\frac{u^+}{\lambda}\right)\,vdx < \infty.
\end{equation}
By \cite[Theorem 8.16]{MR924157} and a change of variables, we can rewrite this integral and estimate as follows:
\begin{align*}\int_0^\infty H'(t)v(\{u>\lambda t\})\,dt 
&= \frac{1}{\lambda}\int_0^\infty H'\left( \frac{s}{\lambda}\right)v(\{u>s\})\,ds \\
&\leq \frac{v(\Omega)}{\lambda}\int_0^{r_1}H'\left(\frac{s}{\lambda}\right)\,ds + \frac{1}{\lambda}\int_{r_1}^\infty H'\left(\frac{s}{\lambda}\right) v(\{u>s\})\,ds \\
&\leq v(\Omega)H\left( \frac{\alpha}{\lambda}\right) 
+ \frac{1}{\lambda}\sum_{k=1}^\infty \int_{r_k}^{r_{k+1}} H'\left(\frac{s}{\lambda}\right) v(\{u>r_k\})\,ds \\
&\leq  v(\Omega)H\left( \frac{\alpha}{\lambda}\right) 
+ \frac{\alpha}{\lambda}
\sum_{k=1}^\infty H\left(\frac{r_{k+1}}{\lambda}\right) e^{-p(k-1)}\mu_1 \\
& = v(\Omega)H\left( \frac{\alpha}{\lambda}\right) 
+ \frac{\alpha\mu_1}{\lambda}\sum_{k=1}^\infty \left(\frac{r_{k+1}}{\lambda} \right)^r \exp\left(\frac{\alpha}{\lambda}(k+1) - p(k-1)\right).
\end{align*}
If we choose $\lambda>0$ so that $p-\epsilon=\frac{\alpha}{\lambda}$ for some $\epsilon>0$, then the exponent becomes
\[(p-\epsilon)(k+1) -p(k-1)= (2p-\epsilon) -\epsilon k.  \]
Using this, we can rewrite the final term above as
\begin{equation*}
v(\Omega)H\left( \frac{\alpha}{\lambda}\right) 
+\mu_1 e^{2p-\epsilon} \left(\frac{\alpha}{\lambda} \right)^{r+1}\sum_{k=1}^\infty (k+1)^r e^{-\epsilon k} 
=\Lambda<\infty.
\end{equation*}
Thus, for our chosen value of $\lambda>0$ we get that
 \[ \int_\Omega H\left(\frac{u^+}{\lambda}\right)\,vdx \leq \Lambda.  \]
Note that, taking into account the dependencies of the constants $\alpha$ and $\lambda$, we have that
$\Lambda = \Lambda(H,t,p,S_N,\Omega,v)$. We may assume without loss of generality that $\Lambda>1$, and so by the convexity of $H$ we have that
 \[ \int_\Omega H\left(\frac{u^+}{\lambda\Lambda }\right)\,vdx 
\leq  \frac{1}{\Lambda}\int_\Omega H\left(\frac{u^+}{\lambda}\right)\,vdx \leq 1.\]
Hence, by the definition of the Orlicz norm, 
\[ \|u^+\|_{L^H(v,\Omega)}\leq \lambda \Lambda. \]

If $0<\|f\|_{L^\infty(v,\Omega)}+\|g\|_{L^\infty(v,\Omega)}<\infty$, the homogeneity of our equation ensures that 
\[ w=\frac{u}{2\left(\|f\|_{L^\infty(v,\Omega)}+\|g\|_{L^\infty(v,\Omega)}\right)}\] 
is a degenerate weak subsolution of the Dirichlet problem \eqref{dirichlet-problem} with data
\[ \tilde{f}=\frac{f|f|^{p-2}}{2^{p-1}\left(\|f\|_{L^\infty(v,\Omega)}
+\|g\|_{L^\infty(v,\Omega)}\right)^{p-1}} \quad \text{and} \quad 
\tilde{g} = \frac{g|g|^{p-2}}{2^{p-1}\left(\|f\|_{L^\infty(v,\Omega)}
+\|g\|_{L^\infty(v,\Omega)}\right)^{p-1}}.\]  
Since $\|\tilde{f}\|_{L^\infty(v,\Omega)}+\|\tilde{g}\|_{L^\infty(v,\Omega)}\leq 1$, our previous estimate applied to $w$  gives
\[\|u^+\|_{L^H(v,\Omega)} \leq 2\lambda \Lambda \left(\|f\|_{L^\infty(v,\Omega)}+\|g\|_{L^\infty(v,\Omega)}\right).\] 

Finally, if $\|f\|_{L^\infty(v,\Omega)}+\|g\|_{L^\infty(v,\Omega)}=0$, it suffices to prove that the only nonnegative subsolution to the homogenous equation is the $0$ function.  Let $(u,\grad u) \in QH^{1,p}_0(v,\Omega)$ be a subsolution and test it against $(u^+,\chi_{u>0}\nabla u)$ in the definition, which immediately yields, since $f=g=0$ $v$-a.e., that
\[ \int_{S(0)} |\sqrt{Q}\grad u|^p\,dx \leq  0.  \]
Hence, $\grad u=0$ in $QL^p(S(0))$.  But then, by the Sobolev inequality, we have that 
\[ \|u^+\|_{L^p(v,\Omega)}\leq S_N\|\sqrt{Q} \grad u\|_{L^p(S(0))} = 0, \]
and so $u^+=0$ $v$-a.e. \\


\medskip

We now describe how to modify the proof of Theorem~\ref{thm:no-gain} to prove Corollary~\ref{cor:endpt-powergain}. In this case we start directly with~\eqref{eqn:next-step} and the definition of $D$ used in the proof of Theorem~\ref{thm:power-gain}.  Then, with the notation above, instead of \eqref{2-2-foundation} we get that
\[ \mu_{k+1} \leq e^{-\sigma p} \mu_k \log(\mu_k^{-1})^{-\sigma p} \leq e^{-\sigma p} \mu_k. \]
Given this, we can now argue exactly as we did above, replacing the $L^\infty$ norms of $f$ and $g$ with the $L^A$ and $L^B$ norms, respectively, to prove the desired estimate. 

\section{Applications to specific geometric settings}
\label{section:geometric}

In this section we give some applications of our main results to a specific geometric setting considered by Korobenko, {\em et al.}~\cite{korobenko2024,Korobenko:2016ue}.  We begin with an elementary proposition which generalizes the classical argument showing how the classical Sobolev inequality can be proved from the endpoint case when $p=1$.  

\begin{prop} \label{prop:1-pSobolev}
    Let $\Omega\subset \R^n$ be a bounded domain and fix $1<p<\infty$.  Suppose $Q$ is a matrix such that $|Q|_\op^{\frac{p}{2}} \in L^1(\Omega)$.  Suppose that there exists a Young function $A$ such that the degenerate Sobolev inequality
    \begin{equation} \label{eqn:1-sobolev}
 \|u\|_{L^A(\Omega)} \leq S_A \|\grad u\|_{QL^1(\Omega)} 
 \end{equation}
    holds for all $u \in \lip_0(\Omega)$.  Define $B_p(t)=A(t^p)$.  Then for every $u \in \lip_0(\Omega)$,
    \begin{equation}  \label{eqn:p-sobolev}
 \|u\|_{L^{B_p}(\Omega)} \leq S_{B_p} \|\grad u\|_{QL^p(\Omega)}. 
 \end{equation}
\end{prop}

\begin{proof}
Fix $1<p<\infty$ and $u\in \lip_0(\Omega)$.  If we let $w=\sgn(u)|u|^p$, then $w\in \lip_0(\Omega)$ and $\grad w = \sgn(u)p|u|^{p-1}\grad u$.  Therefore, by rescaling and inequality \eqref{eqn:1-sobolev}, we have that
\begin{multline*}
    \|u\|_{L^{B_p}(\Omega)}^p
     = \||u|^p\|_{L^A(\Omega)} 
      = \|w\|_{L^A(\Omega)} 
      \leq C_A \|\grad w\|_{QL^1(\Omega)}
      \leq  C_A\int_\Omega p |u|^{p-1} |\sqrt{Q}\grad u|\,dx \\
      \leq p C_A\bigg(\int_\Omega |u|^{(p-1)p'}\,dx\bigg)^{\frac{1}{p'}}
     \bigg(\int_\Omega |\sqrt{Q}\grad u|^p\,dx\bigg)^{\frac{1}{p}} 
      =  p C_A\|u\|_{L^p(\Omega)}^{\frac{p}{p'}}\|\grad u\|_{QL^p(\Omega)}. 
      \end{multline*}
Since $t^p \lesssim B_p(t)$, by Lemma~\ref{normcompare} we have that
\[ \|u\|_{L^{B_p}(\Omega)}^p
      \leq C(p,A,\Omega) \|u\|_{L^{B_p}(\Omega)}^{p-1}\|\grad u\|_{QL^p(\Omega)}. \]

If we rearrange terms we get \eqref{eqn:p-sobolev}.
\end{proof}

\begin{remark}
    If we assume that $v=1$ and $Q$ is such that $|Q|_\op \leq k$, then Proposition~\ref{prop:1-pSobolev} holds for all $p$.  The fact that we take $v=1$ seems central to the above proof; it is an open problem to prove a comparable result assuming  $v\in L^1(\Omega)$ and $|Q|_\op \leq kv$ as we did in our main results. 
\end{remark}

\begin{remark} \label{remark:our-sobolev}
    As a special example, if $A(t) \approx t\log(e+t)^\sigma$ for some $\sigma >0$ in \eqref{eqn:1-sobolev}, then Proposition \ref{prop:1-pSobolev} gives \eqref{eqn:p-sobolev} with $B_p(t) \approx t^p\log(e+t)^\sigma$, the main hypothesis of Theorems~\ref{thm:orlicz-gain} and~\ref{thm:orlicz-gain-g=0} when $v=1$ and $|Q|_\op\leq k$.
\end{remark}

\medskip

With Proposition~\ref{prop:1-pSobolev} we can now compare our results to those in~\cite{korobenko2024,Korobenko:2016ue}, which we will restate here using our notation, which is somewhat different than theirs.   In both of these papers they are interested in matrices $Q$ that are bounded and infinitely degenerate; that is, they assume $v=1$ and $|Q|_\op \in L^\infty(\Omega)$.  Their results require that $Q$ have a particular diagonal form.  They also require careful control of the degeneracy of smallest eigenvalue, and so they assume a number of technical hypotheses on it.  To avoid getting bogged down in the details, for each paper we will simply refer to these hypotheses as the "$F$-geometry" (which is their terminology) and we refer the reader to~\cite[Definition~16]{Korobenko:2016ue} and~\cite[Definition~32]{korobenko2024} (as well as the statements of the results cited below) for complete details.   

In~\cite{Korobenko:2016ue}, the authors consider the equation $-\Div(Q\grad u)=f$.  
In~\cite[Proposition~81]{Korobenko:2016ue} they prove that assuming the $F$-geometry, they get an Orlicz Sobolev inequality:  for all $\varphi \in \lip_0(\Omega)$,
\begin{equation} \label{eqn:sawyer-sobolev1} 
\|\varphi\|_{L^N(\Omega)} \leq S_N\|\grad \phi\|_{QL^1(\Omega)}, 
\end{equation}
where $N(t)\approx t\log(e+t)^\sigma$, $\sigma>1$.  With this Sobolev inequality, they prove (see~\cite[Theorem~83]{Korobenko:2016ue}) a somewhat more general version of Theorem~\ref{thm:orlicz-gain-g=0} when $p=2$ (that is, that  solutions are bounded on $\Omega$), assuming that the data $f$ satisfies 
\begin{equation} \label{eqn:dualf}
\sup_{\varphi \in \lip_0(\Omega)} \frac{\displaystyle \bigg|\int_\Omega f(y)\varphi(y)\,dy\bigg|}{\|\grad \varphi\|_{QL^1(\Omega)}} <\infty.  
\end{equation}
If we use the Sobolev inequality~\eqref{eqn:sawyer-sobolev1} to estimate the denominator, and apply  H\"older's inequality~\eqref{holders} with Young function $N$, we see that this condition is satisfied if $f\in L^{\bar{N}}(\Omega)$, where $\bar{N}(t)\approx \exp(t^{\frac{1}{\sigma}})-1$.  

In comparison:  by Remark~\ref{remark:our-sobolev}, \eqref{eqn:sawyer-sobolev1} implies  the Sobolev inequality 
\begin{equation} \label{eqn:our-sobolev} \|\varphi\|_{L^{N_2}(\Omega)}\leq S_{N_2}\|\grad \varphi \|_{QL^2(\Omega)}, 
\end{equation}
where $N_2(t)= t^2\log(e+t)^\sigma$, $\sigma>1$.  Therefore, we can apply Theorem~\ref{thm:orlicz-gain-g=0} to show our solutions are bounded, but with the assumption that our data $f$ satisfies $f\in L^A(\Omega)$, where $A(t)=\exp(t^\alpha)-1$, $\alpha>{\frac{1}{\sigma-1}}$.  Thus, while we assume a somewhat weaker Sobolev inequality, we require stronger exponential integrability on the function $f$.  We also note that with additional assumptions they proved that solutions are not just bounded but continuous (see\cite[Theorem~84]{Korobenko:2016ue}). 

\medskip

In~\cite{korobenko2024}, the authors consider the equation $-\Div(Q\grad u)=f+\Div(\vecg)$ in $\R^2$.  Assuming the $F$-geometry, they prove a local, modular Sobolev inequality (see~\cite[Proposition~37]{korobenko2024}):
\begin{equation} \label{eqn:sawyer-sobolev2}
    N^{-1}\bigg(\avgint_{B(r)} N(|\varphi(y)|)\,dy\bigg) 
    \leq S_N(r) \|\grad \varphi\|_{QL^1(|B(r)|^{-1}dy,B(r))}, 
\end{equation}
where $B(r)$ is a metric ball of radius $r>0$ in the Carnot-Caratheodory metric induced by $Q$, and for $t>0$ sufficiently large,
\begin{equation} \label{eqn:funnyN}
N(t) = \exp\big( [ \log(t)^{\frac{1}{m}}+1]^m\big) 
\end{equation}
with $m>1$.  Using this, they prove that solutions are locally bounded on metric balls, assuming that $f$ satisfies~\eqref{eqn:dualf} with $\Omega=B(r)$, and that $\vecg$ satisfies
\begin{equation} \label{eqn:dualg}
    \sup_{\varphi \in \lip_0(B(r))} \frac{\displaystyle \bigg|\int_{B(r)} \sqrt{Q(y)}\grad \varphi(y) \cdot \vecg(y)\,dy\bigg|}{\|\grad \varphi\|_{QL^1(B(r))}} <\infty. 
\end{equation}

Arguing as before, if we assume that the norm Sobolev inequality corresponding to~\eqref{eqn:sawyer-sobolev2} holds, then~\eqref{eqn:dualf} holds if $f\in L^{\bar{N}(B_r)}$, where $N$ is given by \eqref{eqn:funnyN}. 
Clearly, \eqref{eqn:dualg} holds if $|\vecg|$ is in $L^\infty(B(r))$; it is not obvious whether it holds for any unbounded $\vecg$.  In~\cite[Proposition~17]{korobenko2024}, the authors note these estimates, but while they explicitly assume the modular Sobolev inequality, in their proof they appear to be using the norm Sobolev inequality.  Since $N$ is submultiplicative, the modular Sobolev inequality \eqref{eqn:sawyer-sobolev2} is weaker than the norm Sobolev inequality with this $N$ (see~\cite[Lemma~16]{korobenko2024}).  

It is an interesting open question whether their proof of~\eqref{eqn:sawyer-sobolev2} can be modified to prove a norm Sobolev inequality.  If it can, then our proof of Theorem~\ref{thm:orlicz-gain} could be modified to work with Young functions of the form
$N_2(t)\approx \exp\big( [ \log(t^2)^{\frac{1}{m}}+1]^m\big)$,  which would yield results better than we proved in Theorem~\ref{thm:orlicz-gain}:  since $t^2\log(e+t)^\sigma \lesssim N_2(t)$ for all $\sigma>0$, we would get correspondingly weaker assumptions on the data functions $f$ and $g$.  (This follows from the fact that if $A$ and $B$ are Young functions and $A(t) \lesssim B(t)$, then $\bar{B}(t) \lesssim \bar{A}(t)$.)

On the other hand, by using the Sobolev inequality~\eqref{eqn:our-sobolev} in Theorem~\ref{thm:orlicz-gain}, we get global boundedness results on $\Omega \subset \R^n$ for all $n\geq 2$, and we may assume that the data function $\vecg$ is exponentially integrable and so unbounded.  

\bibliographystyle{plain}
\bibliography{references}
\end{document}